\documentclass{amsart}
\usepackage{amsthm, amsfonts, amssymb, amsmath,color}

\newtheorem{theorem}{Theorem}[section]
\newtheorem{lemma}[theorem]{Lemma}
\newtheorem{corollary}[theorem]{Corollary}
\newtheorem{proposition}[theorem]{Proposition}
\newtheorem{remark}[theorem]{Remark}

\theoremstyle{definition}

\newtheorem{remark/example}[theorem]{Remark/Example}

\let\oldlabel=\label
\def\prellabel{\marginparsep=1em\marginparwidth=44pt
 \def\label##1{\oldlabel{##1}\ifmmode\else\ifinner\else
 \marginpar{{\footnotesize\ \\ \tt
 ##1}}\fi\fi}}

\numberwithin{equation}{section}

\def\PP{ {\bf P} }
\def\NN{ {\bf N} }
\def\ZZ{ {\bf Z} }

\def\FF{ {\bf F} }
\def\GG{ {\bf G} }

\newcommand{\Rees}{\operatorname{Rees}}

\newcommand{\GL}{\operatorname{GL}}
\newcommand{\PGL}{\operatorname{PGL}}

\newcommand{\Tor}{\operatorname{Tor}}

\newcommand{\reg}{\operatorname{reg}}

\newcommand{\Aldo}[1]{} 

\newcommand{\Giulio}[1]{}

\usepackage{pdfsync}

\numberwithin{equation}{section}
 
\begin{document}
\title{Koszul property of projections of  the Veronese cubic surface}
\author{Giulio Caviglia}
\address{ Department of Mathematics, Purdue University, 150 N. University Street, West Lafayette, IN 47907-2067, USA
}
\email{gcavigli@math.purdue.edu }
\thanks{The work of the first author was supported by a grant from the Simons Foundation (209661 to G. C.)}

\author{Aldo Conca}
\address{Dipartimento di Matematica, Universit\'{a} di Genova, Via Dodecaneso 35, 16146 Genova,
Italy} 
\email{conca@dima.unige.it}

\dedicatory{To Tito Valla, our teacher and friend}

 \date{\today}

\begin{abstract}
Let $V\subset {\bf P}^9$ be the Veronese cubic surface. We classify the projections  of $V$  to ${\bf P}^8$ whose coordinate rings are Koszul. In particular we obtain a purely theoretical proof of the Koszulness
of the pinched Veronese, a result obtained originally by Caviglia
using  filtrations, deformations and computer assisted computations.
To this purpose we extend, to certain complete intersections,
results of Conca, Herzog, Trung and Valla concerning homological
properties of diagonal algebras.

\end{abstract}

\maketitle  
  
 \section{Introduction} 
 Koszul algebras were originally introduced by Priddy \cite{P} in his study of  homological properties of graded (non-commutative) algebras arising from various constructions in algebraic topology.  Given a field $K$,  a positively graded $K$-algebra $A=\oplus_{i\in \NN} A_i$ with $A_0=K$ is Koszul if the field $K$,  viewed as a $A$-module via the identification $K=A/A_+$, has a linear free resolution. In the very interesting volume \cite{PP}  Polishchuk and Positselski  discuss various surprising aspects of Koszulness. 
 
In the commutative setting Koszul algebras can be characterized by means of  the relative Castelnuovo-Mumford regularity. Paraphrasing Hochster \cite[p. 887]{H}, one can say that life is worth living in standard graded algebras when all the finitely generated graded modules have finite (relative) Castelnuovo-Mumford regularity. Such algebras are exactly the commutative Koszul algebras, see Avramov and Eisenbud \cite{AE} and Avramov and Peeva \cite{AP}. 

Let $K$ be a field and $R=K[x_0,x_1,x_2]$.  The pinched Veronese is the $K$-subalgebra of $R$ generated by all the monomials of degree $3$ with the exception of $x_0x_1x_2$. It can be seen as the coordinate ring of a projection  from a point  of the Veronese cubic surface $V_{2,3}$ of $\PP^9$, that is,   the embedding  of  $\PP^2$ in $\PP^9$ with the forms of degree $3$.  It is a ``generic" projection with respect to the stratification of the ambient space by the secant varieties of $V_{2,3}$. That is,  the center of the projection is outside  the second secant variety $\sec_2(V_{2,3})$ of $V_{2,3}$ and the third secant is the ambient space $\PP^9$.  

 In the nineties Sturmfels asked,  in a conversation with Peeva,  whether the pinched Veronese is Koszul and the problem became quickly known as a benchmark example to test new theorems and techniques.   The first  author of the present  paper proved in \cite{C} that the pinched Veronese is indeed Koszul by using a combination of arguments based on filtrations, deformations and computer assisted computations.  

More generally, one can ask the same question for any projection of $V_{2,3}$ to $\PP^8$. In particular,  one can ask whether  the projection of  $V_{2,3}$ to $\PP^8$ from a point that  does not belong  to  $\sec_2(V_{2,3})$ is Koszul. 
The goal of the paper is to show that this is indeed the case. As a special case, we obtain a entirely theoretical proof of the Koszulness of the pinched Veronese.

To achieve this result we develop in Section \ref{se1}  homological arguments that generalize results of Conca, Herzog, Trung and Valla \cite{CHTV}.   
Given a standard $\ZZ^2$-graded $K$-algebra $S$ and  a cyclic  subgroup $\Delta$ of $\ZZ^2$   one considers the ``diagonal" subalgebra $S_\Delta$  of $S$ defined as $\oplus_{v\in \Delta} S_v$.  Similarly, for every $\ZZ^2$-graded $S$-module $M$ one defines the $S_\Delta$-module $M_\Delta$ as $\oplus_{v\in \Delta} M_v$.  For every element $w\in \ZZ^2$, the shifted copy $S(w)$ of $S$ is defined as the $\ZZ^2$-graded $S$ module whose $v$-th component is $S_{v+w}$.   
 
In the transfer of homological information  from $S$ to $S_\Delta$ it is crucial to bound the homological invariants of the  shifted-diagonal modules  $S(w)_\Delta$ as $S_\Delta$-modules.  
When $S$ is is a bigraded polynomial ring, it is proved 
in \cite{CHTV} that the modules  $S(w)_\Delta$ have  a linear $S_\Delta$-free resolution.  
We extend this result to the main diagonal $\Delta=(1,1)\ZZ$ of
certain bigraded complete intersections, see Section  \ref{se1}.

 In Section \ref{se2} we  prove that, given a  complete intersection $I$ of $3$ quadrics in a  polynomial ring $R$,  the $K$-subalgebra of $R$ generated by the cubics in $I$ is Koszul. This is done by constructing a complex whose homology vanishes along the relevant diagonal. 
 Finally in Section \ref{se3} we reinterpret the result of Section \ref{se2} to get the classification of the projections to $\PP^8$ 
 of the cubic Veronese surface  with a Koszul coordinate ring.  
 
  We thank Alexandru Constantinescu, Giorgio Ottaviani, Euisung Park, Claudiu Raicu,  Takafumi Shibuta and Peter Schenzel  for several useful discussions concerning the various aspects of the paper. 
 
 \section{Generalities and preliminary results} 
 \label{se1}
Let $K$ be a field and $A$ be a standard graded $K$-algebra, that is, a commutative algebra of the form $A=\oplus_{i\in \NN} A_i$ such that $A_0=K$, $\dim_K A_1$ is finite and $A$ is generated as a $K$-algebra by $A_1$. In other words, a standard graded $K$-algebra can be written as $A=S/I$ where $S=K[x_1,\dots,x_n]$ is a polynomial ring over $K$  equipped with the graded structure induced by the assignment $\deg(x_i)=1$ for every $i$ and $I$ is a homogeneous ideal.  The algebra $A$ is said to be quadratic if its defining ideal $I$ is generated by quadrics and $G$-quadratic if $I$ has a Gr\"obner basis of quadrics (with respect to some coordinate system of $S$ and some term order).  
  
  Given a non-zero graded $A$-module $M$ one  defines the (relative) Castelnuovo-Mumford regularity as
$$\reg_A(M)= \sup\{ j-i : \Tor_i^A(M,K)_j\neq 0\}\in \ZZ\cup \{+\infty\}.$$  
One says that $M$ has a linear $A$-resolution if $M$ is generated in a single degree, say $d$,  and $\reg_A(M)=d$.  Also, $A$ is Koszul if the residue field  $K$ has a linear $A$-resolution, i.e. $\reg_A K=0$. 
It turns out that  a standard graded $K$-algebra $A$ is Koszul if and only if  $\reg_A(M)$ is finite for every finitely generated graded $A$-module $M$, see \cite[Theorem 1]{AE} and \cite[Theorem 2]{AP}. 
 
It is known that Koszul algebras are quadratic and that G-quadratic algebras are Koszul.  Both implications are, in general, strict, see Eisenbud, Reeves and Totaro \cite{ERT}. 
 
One defines  the graded Poincar\'{e} series of $M$ as 
$$P_M^A(t,s)=\sum_{i, j} \dim_K \Tor_i^A(M,K)_j s^jt^i\in \ZZ[s][|t|].$$

\begin{lemma}\label{Cartan-Eilenberg}
Let $A$ be a standard graded $K$-algebra and let $B$ be a graded  quotient of $A$. Let $M$ be a finitely generated graded $B$-module. 
Assume $\reg_A(B)\leq 1$. Then: 
\begin{itemize} 
\item[(1)] $\reg_B(M) \leq \reg_A(M)$. 
\item[(2)]  If $M$ has a $A$-linear resolution then $M$ has a $B$-linear resolution. 
\item[(3)]  If  $A$ is Koszul then $B$ is Koszul as well. 
\end{itemize}
\end{lemma}

\begin{proof} Notice that (3) is a special case of $(2)$ and  (2) is a special case of (1). To prove (1) we may assume $\reg_A(M)$ is finite (otherwise there is nothing to prove), say $a=\reg_A (M)$.  
The Cartan-Eilenberg spectral sequence,  in the graded setting, induces a coefficientwise  inequality
$$P_M^B(t,s)\leq  \frac{P_M^A(t,s)}{1+t-tP_B^A(t,s)},$$
see Avramov \cite[Proposition 3.3.2]{A}. Set  $G(t,s)=tP_B^A(t,s)-t$ so that,  by assumption, we may write $G(t,s)=\sum_{i\geq 2} g_i(s)t^i$ with $g_i(s)$ polynomials of degree $\leq i$. 
It follows that the terms $s^jt^i$ that appear with a  non-zero coefficient in $1/(1-G)=\sum_{k\geq 0} G^k$  satisfy $j\leq i$. 
But then the terms $s^jt^i$ that appear with a  non-zero coefficient in $P_M^A(t,s)/(1-G)$ satisfy $j\leq i+a$ and the desired inequality follows. 
\end{proof} 

Let $R=\oplus_{(i,j)\in \ZZ^2}   R_{(i,j)} $ be a bigraded  standard $K$-algebra. Here standard means that $R_{(0,0)}=K$ and that $R$ is generated as a $K$-algebra by the $K$-vector spaces $R_{(1,0)}$ and $R_{(0,1)}$ of finite dimension.   Let  $\Delta$ be the diagonal $(1,1)\ZZ\subset \ZZ^2$. We set $R_\Delta=\oplus_{i\in \ZZ}  R_{(i,i)}$ and observe that $R_\Delta$ is the $K$-subalgebra of $R$ generated by $R_{(1,1)}$ and hence it is a standard graded $K$-algebra. For every bigraded $R$-module $M=\oplus_{(i,j)\in \ZZ^2} M_{(i,j)}$ we set  $M_\Delta=\oplus_{i\in \ZZ} M_{(i,i)}.$ Notice that $M_{\Delta}$ is a module over  $R_\Delta$.  We may think of $-_\Delta$ as a functor from the category of bigraded $R$-modules and maps of degree $0$ to that of graded $R_\Delta$-modules with maps of degree $0$. Clearly $-_\Delta$, being a selection of homogeneous components,  is an exact functor.  

 For $(a,b)\in \ZZ^2$ let $R(-a,-b)$ be the shifted copy of $R$. The diagonal module $R(-a,-b)_\Delta$ is a $R_\Delta$-submodule of $R$. The homological  properties of  $R(-a,-b)_\Delta$  play an important role in the transfer of homological information  from $R$ to $R_\Delta$. Note that $R(-a,-b)_\Delta$ is generated by $R_{(0,a-b)}$ if $a\geq b$ and by $R_{(b-a,0)}$ if $b\geq a$. More precisely, one has:  
 $$R(-a,-b)_\Delta=\left\{ 
 \begin{array}{llll}
 R(0,a-b)_\Delta(-a) & \mbox{ if } a\geq b\\
 R(b-a,0)_\Delta(-b) & \mbox{ if } b\geq a.
 \end{array}
 \right.
 $$
 We have the following: 
\begin{proposition} \label{messi}
Let $S=K[x_1,\dots,x_m, t_1,\dots, t_n]$ be the   polynomial ring bigraded by $\deg x_i=(1,0)$ for $i=1,\dots,m$ and $\deg t_i=(0,1)$  for $i=1,\dots,n$.  Let $I$ be an ideal of $S$ generated by a regular sequence  of elements of $S$ all of bidegree $(2,1)$ and $R=S/I$. Then: 
\begin{itemize} 
\item[(1)] $R_\Delta$ is Koszul. 
\item[(2)] For every $(a,b)\in \ZZ^2$ the module $R(-a,-b)_\Delta$ has a linear $R_\Delta$-resolution, i.e., $\reg_{R_\Delta} R(-a,-b)_\Delta=\max(a,b)$.  
\end{itemize} 
 \end{proposition} 

 \begin{proof} Let $h$ be the codimension of $I$. We argue by induction on $h$. If  $h=0$ then $R_\Delta$  is the Segre product of $K[x_1,\dots,x_m]$ and $K[t_1,\dots, t_n]$ that  is Koszul, indeed G-quadratic, see for instance \cite{ERT}. In this case, statement (2) is proved in \cite[proof of Thm 6.2]{CHTV}. 
 Assume $h>0$. We may write $R=T/(f)$ where $f$ is a $T$-regular element of bidegree $(2,1)$ and where $T$ is defined by a $S$-regular sequence of length  $h-1$ of elements of bidegree $(2,1)$. We have a short exact sequence of $T$-modules: 
\begin{equation}\label{eq1} 
0\to T(-2,-1)\to T \to R\to 0
\end{equation}
and applying $\Delta$ we have an exact sequence of $T_\Delta$-modules: 
 $$0\to T(-2,-1)_\Delta \to T_\Delta  \to R_\Delta \to 0.$$ 
 By induction we know $T_\Delta$ is Koszul and that $\reg_{T_\Delta} T(-2,-1)_\Delta=2$.  Hence 
 $$\reg_{T_\Delta} R_\Delta \leq 1.\label{eq2}$$ 
 By Lemma \ref{Cartan-Eilenberg}(3) we may conclude that $R_\Delta$ is Koszul.  
 In order to prove (2) we divide the discussion in three cases.
\medskip 

\noindent  {\bf Case 1.} $a=b$. Just observe that $R(-a,-a)_\Delta=R_\Delta(-a)$. 
 
  \medskip 

\noindent  {\bf Case 2.} $b>a$. We first shift  (\ref{eq1}) by $(-a,-b)$ and then apply $\Delta$.  We get a short exact sequence of  $T_\Delta$-modules: 
 $$0\to T(-a-2,-b-1)_\Delta \to T(-a,-b)_\Delta  \to R(-a,-b)_\Delta \to 0.$$ 
 So we have: 
 $$\reg_{T_\Delta} R(-a,-b)_\Delta \leq \max\{\reg_{T_\Delta}  T(-a,-b)_\Delta , \reg_{T_\Delta}  T(-a-2,-b-1)_\Delta-1\}.$$ 
 By induction and since we are assuming $b>a$ we deduce that 
 $$\reg_{T_\Delta} R(-a,-b)_\Delta =b.$$ 
 Since we have shown already that $\reg_{T_\Delta} R_\Delta \leq 1$,  by Lemma \ref{Cartan-Eilenberg}(1) we can deduce that $\reg_{R_\Delta} R(-a,-b)_\Delta =b$.
  \medskip

\noindent {\bf Case 3.} $a>b$.
Set $P=(t_1,\dots,t_n)\subset S$.  As already observed we have 
$R(-a,-b)_\Delta= R(0,a-b)_\Delta(-a)$. 
So we have to prove that $\reg_{R_\Delta}R(0,u)_\Delta=0$ for every $u>0$. 
We consider the minimal free (bigraded) resolution of $S/P^u$ as an $S$-module (that is a Eagon-Northcott complex) 
$$\FF: 0\to F_n\to F_{n-1} \to \dots \to F_1\to F_0\to 0$$
with $F_0=S$ and $F_i=S^\#(0,-u-i+1)$ for $i>0$ where $\#$ denotes some integer depending on $n,u$ and $i$ that  is irrelevant in our discussion. 
The homology of $\FF\otimes R$ is $\Tor^S(S/P^u,R)$. We may  as well compute $\Tor^S(S/P^u,R)$  as the homology of  $S/P^u\otimes \GG$ where $\GG$ is a free resolution of $R$ as an $S$-module. By assumption, we may take  $\GG$  to be  the Koszul complex on a sequence of  elements of bidegree $(2,1)$. 
It follows that: 

\[H_i(\FF\otimes R)=\left\{
\begin{array}{lll}
 \mbox{a subquotient of } (S/P^u)^{\#} (-2i,-i)  & \mbox{ if } & 0\leq i\leq h\\
0 &  \mbox{ if  } & i>h.
\end{array}
\right.
\]
Shifting with $(0,u)$ and applying $\Delta$ we have a complex $(\FF\otimes R(0,u))_\Delta$ that, we claim, has no homology at all. 
Shifting and  applying  $\Delta$ are compatible operations  with taking homology. Therefore to prove that $(\FF\otimes R(0,u))_\Delta$ has no homology we have only to check that 
$$[(S/P^u)(-2i,-i+u)]_\Delta=0$$
for all $i$ and that  is obvious by degree reasons. 
So we have an exact complex of $R_\Delta$-modules: 
\[0\to  R^{\#}(0,-n)_\Delta\to \dots  \to R^{\#}(0,-1)_\Delta          \to  R^{\#}_\Delta \to          R(0,u)_\Delta \to 0.\]
Since we know (by Case 2) that $\reg_{R_\Delta} R(0,-i)_\Delta=i$ we may conclude (see \cite[Lemma 6.3]{CHTV}) that $\reg_{R_\Delta} R(0,u)_\Delta=0$ as desired.  \end{proof} 

\section{Diagonal algebras of cubics forms}
\label{se2}

Let $I$ be a homogeneous complete intersection ideal of codimension $r$   generated by elements of degree $d$ in a polynomial ring $R$ over a field $K$. Let $c,e\in \NN$ and consider the $K$-subalgebra $A_{c,e}$  of $R$ generated by the forms of degree $c+ed$ in $I^e$, i.e. $A_{c,e}=K[(I^e)_{ed+c}]$. 
If one gives the standard  $\ZZ^2$-graded structure to the Rees algebra $\Rees(I)$ of $I$, then the ring $A_{c,e}$ can  be seen as the $(c,e)$-diagonal subalgebra of $\Rees(I)$, that is, $A_{c,e}=\oplus_{i\in \NN} \Rees(I)_{(ic,ie)}$. 

Corollary  6.10 in \cite{CHTV} asserts that  $A_{c,e}$   is quadratic if $c\geq d/2$ and Koszul if $c\geq d(r-1)/r$. The authors of  \cite{CHTV} ask (at p. 900) whether  $A_{c,e}$ is Koszul also for $d/2\leq c <d(r-1)/r$. The first instance of this problem occurs  for $d=2$ and $r=3$, i.e. $3$ quadrics, where  the only possible value for $c$ is $1$. 
The goal of this section is to treat  this case that, as we will see in the following section, is also the crucial case for the classification problem discussed in the introduction.  
 
 Let $I=(g_1,g_2,g_3)$ be a complete intersection of  quadrics in $R=K[x_1,\dots,x_n]$. Consider the Rees algebra 
 $$\Rees(I)=R[g_1t,g_2t,g_3t]\subset R[t]$$ 
 of $I$  with its standard bigraded structure induced by $\deg x_i=(1,0)$ and $\deg g_it=(0,1)$. It can be realized as  a quotient of the polynomial ring 
 $$S=K[x_1,\dots,x_n,t_1,t_2,t_3]$$ bigraded  with $\deg x_i=(1,0)$ and $\deg t_i=(0,1)$, by the ideal $J$ generated by the $2$-minors of 
 $$ 
X=\left(
\begin{array}{ccc}
 g_1 &  g_2 &  g_3  \\
 t_1 &  t_2 &  t_3
\end{array}
\right)
.$$
Let $f_1,f_2,f_3$ be the $2$-minors of $X$ with the appropriate sign, say $f_i$ equals to  $(-1)^{i+1}$ times the minor of $X$ obtained by deleting the $i$-th column. 
Hence 
$$J=I_2(X)=(f_1,f_2,f_3).$$
 The sign convention is chosen so that the rows of the matrix $X$ are syzygies of $f_1,f_2,f_3$. 
 By construction we have  
 $$\Rees(I)_\Delta=K[I_3]=K[ x_ig_j : i=1,\dots,n, \ \ j=1,2,3]$$
  where $\Delta=(1,1)\ZZ\subset \ZZ^2$.   
Our goal is to prove that $K[I_3]$ is Koszul.

For later use we record the following: 

\begin{lemma}\label{easy}
\begin{itemize}
\item[(1)] $f_1,f_2$ form a regular $S$-sequence. 
\item[(2)] $(f_1,f_2):f_3=(g_3,t_3)$.   
\item[(3)] $(f_1,f_2):t_3=J$.   
\item[(4)] $(t_3,f_1,f_2):g_3=(t_1,t_2,t_3)$.
\end{itemize}
\end{lemma}
 
\begin{proof} 
(1):  the ideal $J$ is prime and hence $f_1,f_2$ have no common factors. 

(2): the inclusion $\supseteq$ follows because the rows of $X$ are syzygies of $f_1,f_2,f_3$.  Clearly  $(g_3,t_3)\supseteq  (f_1,f_2)$. Hence the equality follows if one shows that $g_3,t_3, f_3$ is a regular sequence. But that is obvious because the variable $t_3$ does not appear in the polynomials $g_3$ and $f_3$,  and $f_3$ is irreducible being a  minimal generator of a prime ideal. 
 
(3): the inclusion $\supseteq$ follows because the second row of $X$ is a syzygy of $f_1,f_2,f_3$. The equality follows because $J$ is prime and contains $(f_1,f_2)$. 
 
(4): for the inclusion $\supseteq$ simply note that $f_1=-g_3t_2 \mod(t_3)$ and $f_2=g_3t_1 \mod(t_3)$. The other inclusion follows because $(t_1,t_2,t_3)$ is a prime ideal  containing $(t_3,f_1,f_2)$. 
\end{proof}

Set  $B=S/(f_1,f_2)$. Since  $f_1,f_2$ is a  regular $S$-sequence of elements of bidegree $(2,1)$ we may apply to $B$ the results of Proposition \ref{messi}. 
One has also $B/f_3B=\Rees(I)$.  We will prove that:

\begin{theorem} 
\label{main1} 
We have: 
\begin{itemize} 
\item[(1)] $\reg_{B_\Delta} (\Rees(I)_\Delta)=1$, 
\item[(2)]  $\Rees(I)_\Delta$ is Koszul.  
\end{itemize}
\end{theorem} 

Since, by construction,  $K[I_3]=\Rees(I)_\Delta$, as a corollary of Theorem \ref{main1} we have: 

\begin{corollary} 
\label{macao}
Let $I$ be a complete intersection of $3$ quadrics in $K[x_1,\dots, x_n]$. Then $K[I_3]$ is Koszul. 
\end{corollary} 
 
\begin{proof}[ Proof of Theorem \ref{main1}] 
First of all we note that (2) follows from (1) and Lemma \ref{Cartan-Eilenberg} since,  by Proposition \ref{messi},  
we know that $B_\Delta$ is Koszul.   

It remains to prove (1). 
Since $f_3t_3=0$ in $B$ we have a complex
$$\FF:    \cdots  \to     B(-4,-4)\stackrel{t_3}\longrightarrow    B(-4,-3)\stackrel{f_3}\longrightarrow    B(-2,-2)\stackrel{t_3}\longrightarrow   B(-2,-1)\stackrel{f_3}\longrightarrow      B\to 0,$$
i.e. $F_i=B(-i,-i)$ if $i$ is even, $F_i=B(-i-1,-i)$ if $i$ is odd.  
The homology of $\FF$ can be described by using Lemma \ref{easy}: 

$$H_i(\FF)=\left\{ 
\begin{array}{lll}
\Rees(I)  & \mbox{ if }  & i=0\\
0             & \mbox{ if }  &i \mbox{ is even and positive} \\
S/(t_1,t_2,t_3)(-i-3,-i) & \mbox{ if } &i \mbox{ is odd and positive.}
\end{array} 
\right.
$$
The assertion for $i=0$ holds by construction, and for  $i$  even and positive it holds because of Lemma  \ref{easy}(3). 
Finally for $i$ odd and positive  by Lemma \ref{easy}(2) we have
$$H_i(\FF)=(t_3,g_3)/(t_3,f_1,f_2)(-i-1,-i).$$ 
Hence $H_i(\FF)$ is cyclic generated by the residue class of $g_3$ mod $(t_3,f_1,f_2)$ that has degree $(-i-3,-i)$. Using Lemma \ref{easy}(4) and keeping track of the degrees we get the desired result.     
Note that we have $H_i(\FF_\Delta)=H_i(\FF)_\Delta=0$  for every $i>0$ and $H_0(\FF_\Delta)=\Rees(I)_\Delta$. We may deduce from \cite[Lemma 6.3]{CHTV} that 
$$\reg_{B_\Delta}\Rees(I)_\Delta \leq \sup\{ \reg_{B_\Delta} (F_i)_\Delta -i\}.$$
Since $B$ is defined by a regular sequence of elements of bidegree $(2,1)$ we may apply Proposition \ref{messi} and get 
$$ \reg_{B_\Delta} (F_i)_\Delta=
\left\{
\begin{array}{ll}
i    & \mbox{ if $i$ is even} \\  
i+1& \mbox{ if $i$ is odd.} 
\end{array}
\right.
$$
Summing up,  we obtain $\reg_{B_\Delta} (\Rees(I)_\Delta)=1$. 
\end{proof}

\begin{remark} 
\label{remma}
(1) In \ref{main1}(2)  one cannot replace $\Rees(I)$ with a ring of the form $S/I_2(Y)$ where $Y=(y_{ij})$  is a $2\times 3$ matrix with $\deg y_{1j}=(2,0)$ and $\deg y_{2j}=(0,1)$ and $I_2(Y)$ has codimension $2$. 
For instance for  \[H=I_2
\left(\begin{array}{ccc}
 x_1^2 &  x_2^2 &  0 \\
 0          &  t_2 &  t_3
\end{array}
\right)
\]
and $S=K[x_1,x_2,x_3,t_1,t_2,t_3]$ one has that $(S/H)_\Delta$ is not Koszul  as can be checked by using Macaulay 2 \cite{MC2}. 
In other words, in the proof of \ref{main1} the fact that $J$ is a prime ideal plays a crucial role. 

(2) The coordinate ring of the pinched Veronese can be realized  as $S_\Delta/J_\Delta$,  where  
$$J=I_2
\left(\begin{array}{ccc}
 x_1^2 &  x_2^2 &  x_3^2 \\
 t_1     &  t_2      &  t_3
\end{array}
\right).
$$
Within the Segre product $S_\Delta=K[x_it_j : 1\leq i,j\leq 3]$ the ideal $J_\Delta$ has a Gr\"obner deformation to $H_\Delta=(x_1^2t_2t_i, x_1^2t_3t_i, x_2^2t_3t_i : i=1,2,3)$.  In particular, this shows that the pinched Veronese has a nice quadratic Gr\"obner deformation within the Segre ring $S_\Delta$.  But this is, unfortunately, not enough to prove that it is Koszul. 
\end{remark}

\begin{remark} 
With the notation introduced at the beginning of the section,  we have shown above that $A_{1,1}$ is Koszul for $r=3$ and $d=2$. 
The statements and the proofs of Proposition \ref{messi} and Theorem \ref{main1} generalize immediately to the case of diagonal $(1,e)\ZZ$ and one obtains that  $A_{1,e}$ is Koszul as well. Moreover the case $c=0$ is obvious because $A_{0,e}$ is the $e$-th Veronese ring of a polynomial ring in $3$ variables.  So we may conclude that,  for $r=3$ and $d=2$,  $A_{c,e}$ is Koszul for all $c$ and $e$.  \end{remark}

 \section{Projections to $\PP^8$ of the Veronese cubic embedding of $\PP^2$ } 
 \label{se3}
Let $c,n$ be positive integers and $V$ be a vector space of dimension $n+1$ over a field $K$ of characteristic $0$ or $>c$. The Veronese embedding of $\PP(V)$ with the forms of degree $c$ can be identified with the map 
 $$v_{n,c}: \PP(V)\to \PP(S_c(V))$$ sending $[x]$ to $[x^c]$.  Here $S_c(V)$ denotes the degree $c$ component of the symmetric algebra $S(V)$. The coordinate ring of $\PP(V)$ is  $S(V^*)$. Denote by $V_{n,c}$ the image of $v_{n,c}$. 
Denote by $(g,f)\to g\circ f$ the natural bilinear form  
$$S_c(V^*)\times S_c(V)\to K$$
that is $\GL(V)$-equivariant.  By taking a non-zero $F\in S_c(V)$ we may consider the projection $\phi_F$ from $\PP(S_c(V))$ to   $\PP(W)$ where $W=S_c(V)/\langle F\rangle$. The coordinate ring (of the closure of) the image of the composition $\phi_F \circ v_{n,d}$  gets identified with   the $K$-subalgebra of  $S(V^*)$ generated by  $U_F=\{ g\in S_c(V^*) : g\circ F=0\}$,  a space of forms of degree $c$ of codimension $1$.   After fixing a basis $y_0,\dots, y_n$ for $V$ and the dual basis $x_0,\dots, x_n$ for $V^*$, then $S(V)=K[y_0,\dots,y_n]$ and $S(V^*)=K[x_0,\dots,x_n]$. Furthermore  $g(x_0,\dots,x_n)\circ f(y_0,\dots,y_n)$ gets identified with the action of the differential operator  $g(\partial/\partial y_0,\dots,\partial/\partial y_n)$ on $ f(y_0,\dots,y_n)$. 

Returning now to the Veronese cubic embedding $V_{2,3}$ of $\PP^2$ in $\PP^9$,  we may summarize the discussion above as follows. The projection of $V_{2,3}$ to $\PP^8$ from a point $[F]\in \PP^9=\PP(S_3(V))$,  where $S_3(V)=K[y_0,y_1,y_2]_3$,  has coordinate ring $K[U_F]$ where  $U_F=\{ g\in K[x_0,x_1,x_2]_3  : g\circ F=0\}$. The construction is compatible with the  $\PGL_3(K)$ action, hence forms $F$ and $F_1$ that  are in the same $\PGL_3(K)$-orbit will give  projections with  isomorphic coordinate rings.   

Denote by $\sec_i(V_{2,3})$ the $i$-th secant variety of $V_{2,3}$. 
One knows that $\sec_1(V_{2,3})$ is a $5$-dimensional variety defined by the $3$-minors of a $3\times 6$ matrix of linear forms, a catalecticant matrix,  see  Kanev \cite{K} and Raicu \cite{Ra}. 
 It is also known    that  $\sec_2(V_{2,3})$  is a quartic hypersurface, defined by the Aronhold invariant,  that coincides with the Zariski closure of the $\PGL_3(K)$-orbit of the Fermat cubic, see Dolgachev and Kanev \cite{DK} or Ottaviani \cite{O}. 
Furthermore $\sec_1(V_{2,3})\setminus V_{2,3}$ consists of two $\PGL_3(K)$-orbits, the orbit of the following polynomials: 
 $$F_1=y_1y_2^2, \ \    F_2=y_1^3+y_2^3$$ 
 while 
 $\sec_2(V_{2,3})\setminus \sec_1(V_{2,3})$ consists  of  three  orbits, the orbits of:   
 $$F_3=y_1y_0^2+y_2y_1^2, \ \ F_4=y_2^2y_1+y_0^3, \ \ F_5=y_0^3+y_1^3+y_2^3,$$
see \cite[5.13.2]{DK}.  
We have: 

\begin{theorem} 
\label{superT}
Let $F\in K[y_0,y_1,y_2]_3$ and let $A_F$ be the coordinate ring of the projection of $V_{2,3}$ to $\PP^8$ from the point $[F]$.  We have: 
\begin{itemize} 
\item[(1)]  If $[F]\in V_{2,3}$ then $A_F$ is Koszul. More precisely, it is G-quadratic.  
\item[(2)]  If $[F]\in \sec_2(V_{2,3})\setminus V_{2,3}$ then $A_F$ is not quadratic (and hence not  Koszul).  
\item[(3)]  If $[F]\not \in \sec_2(V_{2,3})$ then $A_F$ is  Koszul. 
\end{itemize}
\end{theorem} 
\begin{proof} For $(1)$ we may assume that  $F=y_2^3$ and then $A_F$ is the toric ring generated by all the monomials of degree $3$ in $K[x_0,x_1,x_2]$ different from $x_2^3$. That such a ring is defined by a Gr\"obner basis of quadrics can be proved by direct computations but follows also from the result  \cite[2.13]{D} of De Negri  because the vectors space $U_F$ is a lexicographic segment.  

For $(2)$  we have to check that for each $F_i$ (with $i=1,\dots,5$)   described above the corresponding ring $A_{F_i}$ is not quadratic.   One can compute explicitly  a presentation of  $A_{F_i}$ by elimination using, for instance,   CoCoA \cite{CoCoA} or Macaulay 2 \cite{MC2}.  And then one checks that  there is a cubic form the among generators of the defining ideal. Indeed there is exactly one cubic generator in each case, while the number of quadrics is either  $17$ (for $i=3,5$)  or $18$ (for $i=1,2,4$). That there is a cubic generator in the case of the Fermat cubic $F_5$ has been verified also by Alzati and Russo in \cite[Example 4.5]{AR} by a theoretical argument. Also the case  $F_1$ and $F_2$ can be treated easily because they are ``cones".  For them  it is enough to prove that projections of  the rational normal curve  $V_{1,3}\in \PP^3$ from the corresponding points  are cubic hypersurfaces in $\PP^2$, and this is well-known.  One obtains the cuspidal cubic with $F_1$ and the nodal cubic with $F_2$. 
Finally for $(3)$ one notes that if $[F]\not \in \sec_2(V_{2,3})$ then the ideal $I$ of all the forms $g\in K[x_0,x_1,x_2]$ such that $g\circ f=0$ is a  complete intersection of three quadrics, see Conca, Rossi and Valla  \cite[Cor.6.12]{CRV}. Since $A_F=K[I_3]$, we may apply Corollary \ref{macao} and conclude that $A_F$ is Koszul. 
\end{proof} 

\begin{remark}
In view of Theorem \ref{superT}(3) one can ask whether for  $F\not\in \sec_2(V_{2,3})$  the ring $A_F$ is G-quadratic. It is known that the pinched Veronese does not have a Gr\"obner basis of quadrics in the toric presentation   but we do not know how to exclude a Gr\"obner basis of quadrics  in every other possible coordinate system.  Some of the well-known necessary conditions that a monomial ideal $U  \subset K[z]=K[z_1,\dots,z_9]$ must satisfy to be the initial ideal of the defining ideal of $A_F$ include:
\begin{itemize}
\item[(1)]  $K[z]/U$ must have the same Hilbert series as $A_F$ and the graded Betti numbers at least as big as those of $A_F$, 
\item[(2)]  the radical of $U$ must be pure and connected in codimension $1$,  
\end{itemize} 
see Kalkbrener and Sturmfels \cite{KS} or Varbaro \cite{V}. But there are plenty of quadratic monomial ideals satisfying those conditions 
and we do not know how to exclude that they are  the initial ideal of the defining ideal of the pinched Veronese, not to mention other $A_F$.

\end{remark}

\begin{remark}
For $F\not \in \sec_2(V_{2,3})$ another quite natural question is whether the ring $A_F$ can be deformed,  via a Sagbi deformation,  to the pinched Veronese. As it follows from \cite[Cor.6.12]{CRV}, the answer is positive if and only if   $F$ is singular. So the ``general"  $A_F$ does not have a Sagbi deformation to the pinched Veronese. 
\end{remark}

\begin{remark}
An interesting question suggested by the proof of Theorem \ref{superT} is the following. Suppose $F$ is a form of degree $d$  in $K[y_0,\dots,y_n]$ that  does not involve the variable $y_n$. We may consider the coordinate ring $A$ of the projection of $V_{n,d}$ from $F$ and also the coordinate ring $B$  of the projection  of    $V_{n-1,d}$ from $F$. Is it true that   $A$ is Koszul or quadratic if $B$ is so? That the opposite implication holds true follow easily from the fact that $B$ is an algebra retract of $A$. 
\end{remark}

\end{document}